\documentclass[11pt]{article}

\usepackage[margin=1in]{geometry}
\usepackage[T1]{fontenc}
\usepackage{lmodern}
\usepackage{amsmath,amssymb,amsthm}
\usepackage{xcolor}
\usepackage{graphicx}
\usepackage{tikz}
\usepackage{hyperref}
\hypersetup{hidelinks}
\usetikzlibrary{arrows.meta}

\definecolor{CaltechOrange}{RGB}{255,108,12}
\definecolor{CaltechOrangeLight}{RGB}{255,236,223}
\definecolor{CaltechGray}{RGB}{118,119,123}
\definecolor{CaltechDark}{RGB}{0,59,76}

\newcommand{\rin}{r_{\mathrm{in}}}
\newcommand{\rout}{r_{\mathrm{out}}}
\newcommand{\Aann}{A_{\mathrm{ann}}}
\newcommand{\Adem}{A_{\mathrm{dem}}}

\theoremstyle{plain}
\newtheorem{theorem}{Theorem}
\newtheorem{lemma}[theorem]{Lemma}
\newtheorem{proposition}[theorem]{Proposition}
\theoremstyle{definition}
\newtheorem{definition}[theorem]{Definition}

\title{Optimal pebbling of the hypercube}
\author{Lior Pachter\thanks{The author used GPT-5.5 for assistance with the mathematics and in drafting an initial version of this manuscript. The author developed the proof strategy, verified the mathematical arguments, edited the manuscript, and takes full responsibility for the content.}\\
\small Division of Biology and Biological Engineering\\
\small Department of Computing and Mathematical Sciences\\
\small California Institute of Technology\\
\small Correspondence should be addressed to lpachter@caltech.edu}
\date{\today}

\begin{document}
\maketitle

\begin{abstract}
We prove that the optimal pebbling number of the $n$-dimensional hypercube is
\[
o(Q_n)\,=\,\Theta\!\left(\left(\frac43\right)^n\right).
\]
\end{abstract}

\section*{Introduction}

A pebbling distribution on a graph $G$ assigns a nonnegative integer number of pebbles to
each vertex.  A pebbling move removes two pebbles from one vertex and places one pebble on
an adjacent vertex.  Graph pebbling was introduced by Chung~\cite{Chung}, who considered
pebbling of the $n$-dimensional hypercube $Q_n$; for a survey and general background on
pebbling see~\cite{HurlbertSurvey}.
\label{def:hypercube}

A pebbling distribution is solvable if, for every target vertex, some sequence
of pebbling moves places at least one pebble on the target.  The optimal pebbling number
$o(G)$ was introduced by Pachter, Snevily and Voxman~\cite{PSV}; it is the minimum total
size of a solvable distribution on $G$.
\label{def:pebbling}
The problem of determining $o(Q_n)$ was raised in~\cite{PSV}, and a standard
weight-function argument gives the lower bound
$o(Q_n)\ge (4/3)^n$.  Moews proved the asymptotic upper bound~\cite{Moews}
\[
o(Q_n)\,\le\, (4/3)^{n+O(\log n)}.
\]
Explicit constructions were also studied by Herscovici, Hester, and Hurlbert,
who obtained an $O(1.3763^n)$ bound using product methods with arbitrary target
distributions~\cite{HHH}, and by Herscovici, who obtained an $O(1.34^n)$ bound
using error-correcting codes~\cite{Herscovici}.  Fu, Huang, and Shiue gave the
stronger explicit upper bound~\cite{FHS}
\[
o(Q_n)\,\le\, O\!\left(n^{3/2}(4/3)^n\right).
\]
We prove that the exponential lower bound is sharp up to a constant factor.

\begin{theorem}[Main theorem]\label{thm:main}
The optimal pebbling number of the $n$-dimensional hypercube satisfies
\[
o(Q_n)\,=\,\Theta\!\left(\left(\frac43\right)^n\right).
\]
\end{theorem}

\section*{Lower bound}

Below we present the weight-function lower bound for pebbling for completeness; see
Chung~\cite{Chung} for the hypercube setting, Bunde et al.~\cite{BundeEtAl} for the
optimal-pebbling proof, and the formulation in~\cite{Hurlbert}.

\begin{proposition}[Lower bound]\label{prop:lower}
$o(Q_n)\,\ge\,\left(\frac43\right)^n.$
\end{proposition}

The bound rests on the monotonicity of a weight function under pebbling moves.

\begin{lemma}[Weight monotonicity]\label{lem:weight}
Let $D$ be a pebbling distribution on $Q_n$.  For a target vertex $v$ define the weight
\[
W_v(D)\,=\,\sum_{u\in V(Q_n)} D(u)2^{-d(u,v)}.
\]
Then $W_v$ does not increase under a legal pebbling move, and if $D$ is solvable then
$W_v(D)\ge 1$ for every $v$.
\end{lemma}
\begin{proof}
During a pebbling move the weight cannot increase: replacing two pebbles at $x$ by one
pebble at an adjacent vertex $y$ changes the contribution from $2\cdot 2^{-d(x,v)}$ to
$2^{-d(y,v)}$, and $d(y,v)\ge d(x,v)-1$.  Since $D$ is solvable, some sequence of moves
places a pebble on $v$, whence $W_v(D)\ge 1$ for every $v$.
\end{proof}

\begin{proof}[Proof of Proposition~\ref{prop:lower}]
Let $D$ be any solvable distribution.  Summing the inequality of
Lemma~\ref{lem:weight} over all targets,
\[
2^n
\,\le\, \sum_{v\in V(Q_n)} W_v(D)
\,=\,\sum_{u\in V(Q_n)}D(u)\sum_{v\in V(Q_n)}2^{-d(u,v)}.
\]
For each fixed $u$,
\[
\sum_{v\in V(Q_n)}2^{-d(u,v)}
\,=\,\sum_{i=0}^n \binom ni2^{-i}
\,=\,\left(\frac32\right)^n,
\]
so $2^n\le |D|\left(\frac32\right)^n$ and therefore $|D|\ge \left(\frac43\right)^n$.  Since
$D$ was arbitrary, this proves the lower bound.
\end{proof}

\section*{Upper bound}

\begin{definition}[High-demand number]\label{def:demand}
For an integer $T\ge 1$, let $o_T(Q_n)$ denote the smallest size of a distribution from
which one can move at least $T$ pebbles to any prescribed target.
\end{definition}

\begin{lemma}[Near-optimal high-demand lemma]\label{lem:high-demand}
There are constants $\Adem,C,n_0>0$ such that the following holds.  If $n\ge n_0$,
$n^{-3}\le \delta\le 1/2$, and
\[
T\,\ge\, C\delta^{-2}n2^{\Adem\sqrt{n\log(2/\delta)}},
\]
then
\[
o_T(Q_n)\,\le\, (1+\delta)T\left(\frac43\right)^n.
\]
Moreover, the witnessing distribution may be chosen with every occupied pile of size at
least $2^{\lfloor n/3\rfloor}$.
\end{lemma}
\begin{proof}
Choose a constant $\Aann$ large enough for the binomial tail estimate
below, and set
\[
w\,=\,\left\lceil \Aann\sqrt{n\log(2/\delta)}\right\rceil,\qquad
\rout\,=\,\left\lfloor\frac n3\right\rfloor+w,\qquad
\rin\,=\,\left\lfloor\frac n3\right\rfloor-w,\qquad
S\,=\,2^{\rout}.
\]
The constants $\Aann$ and $\Adem$ control, respectively, the annulus width and the demand
threshold. Later, we choose $\Adem>2\Aann$.  Thus, $\rin$ and $\rout$ are the inner and
outer radii of the annulus around a target.
\label{def:annulus}
Increasing $n_0$ if necessary, the assumption $\delta\ge n^{-3}$ ensures that
$0\le \rin\le \rout\le n$.
We use standard Chernoff bounds~\cite{AlonSpencer}.  By Chernoff's inequality, for
$Y\sim\operatorname{Bin}(n,1/3)$,
\[
\mathbb P(\rin\,\le\, Y\,\le\, \rout)\,\ge\, 1-\delta/8.
\]
Choose
\[
N\,=\,\left\lceil (1+\delta/2)T\frac{(4/3)^n}{S}\right\rceil
\]
centers independently and uniformly from $Q_n$, with replacement, and place $S$ pebbles at
each chosen center.  Repeated centers simply add their piles.
The ceiling changes the cost by at most $S$.  The threshold for $T$ and the assumption
$\delta\ge n^{-3}$ imply, for all sufficiently large $n$,
\[
S\,\le\, \frac{\delta}{2}T\left(\frac43\right)^n,
\]
and therefore
\[
NS\,\le\, (1+\delta)T\left(\frac43\right)^n.
\]

Now fix a target $t$, and count only centers in the annulus $\rin\le d(c,t)\le \rout$:
\[
X_t\,=\,\sum_{j=1}^N 2^{\rout-d(c_j,t)}
{\bf 1}_{\rin\,\le\, d(c_j,t)\,\le\, \rout}.
\]
If a center is at distance $d$ from $t$, its stack of $S=2^{\rout}$ pebbles can send
$2^{\rout-d}$ pebbles to $t$.  Thus the centers counted in $X_t$ can deliver $X_t$
pebbles directly to $t$.  For a uniformly random center, the distance from $t$ has
distribution $\operatorname{Bin}(n,1/2)$. Also
\[
\left(\frac43\right)^n2^{-n}2^{-d}
\,=\,\left(\frac13\right)^d\left(\frac23\right)^{n-d},
\]
so this change of weights converts $\operatorname{Bin}(n,1/2)$ into
$\operatorname{Bin}(n,1/3)$. Therefore,
\[
\mathbb E X_t
\,\ge\, (1+\delta/2)T
\mathbb P\!\left(\rin\,\le\, \operatorname{Bin}(n,1/3)\,\le\, \rout\right)
\,\ge\, (1+\delta/4)T.
\]
Also,
\[
\mathbb E X_t\,\le\, NS\left(\frac34\right)^n\,\le\, (1+\delta)T\,\le\, 2T.
\]
One center contributes at most
\[
B_0\,=\,2^{\rout-\rin}\,=\,2^{2w}
\]
to $X_t$, and the summands defining $X_t$ are independent.  Set
$\mathrm{gap}=\mathbb E X_t-T$; by $\mathbb E X_t\ge (1+\delta/4)T$ we have
$\mathrm{gap}\ge (\delta/4)T$, and $\mathbb E X_t\le 2T$ from the bound on $NS$
above.  A Chernoff bound for the lower tail with bounded increments, optimized
over the exponential parameter, yields
\[
\mathbb P(X_t\,<\,T)\,\le\,
\exp\!\left(-\frac{\mathrm{gap}^2}{4B_0\mathbb E X_t}\right)
\,\le\,\exp\!\left(-c\frac{\delta^2T}{B_0}\right),
\]
valid for $\delta\le 1/2$ with $c=1/128$.  The condition $\Adem>2\Aann$ makes the
demand threshold grow faster than the worst single-center contribution coming
from the annulus width. Indeed, $B_0\le 4\cdot 2^{2\Aann\sqrt{n\log(2/\delta)}}$,
while the hypothesis on $T$ gives
\[
\frac{\delta^2T}{B_0}
  \,\ge\, c' C n\,
  2^{(\Adem-2\Aann)\sqrt{n\log(2/\delta)}}
\]
for an absolute constant $c'>0$.  Taking $C$ sufficiently large makes this
probability at most $2^{-2n}$.  A union bound over all $2^n$ targets then shows that
with positive probability the random distribution is $T$-solvable for every target, so
by the probabilistic method a distribution witnessing the lemma exists.
\end{proof}

The recursion is built on a deterministic composition step, a special case of the product
theorem for solvability with arbitrary target distributions in~\cite{HHH}.  The approach is
illustrated in Figure~\ref{fig:product-step}.  Throughout, we write
$Q_n=Q_a\square Q_m$, where $\square$ denotes the Cartesian product.
\label{def:product}

\begin{lemma}[Product / composition step]\label{lem:product-step}
Suppose $E$ is a solvable distribution on $Q_m$, and for each occupied vertex $z$ of $Q_m$
with $E(z)=s_z$ one has $o_{s_z}(Q_a)\le (1+\delta)s_z\left(\frac43\right)^a$.  Then
\[
o(Q_n)\,\le\, (1+\delta)|E|\left(\frac43\right)^a.
\]
\end{lemma}

\begin{figure}[htbp]
\centering
\begin{tikzpicture}[
  x=1.05cm,
  y=0.72cm,
  stacklabel/.style={text=CaltechGray},
  pebble/.style={circle,draw=black,fill=black,minimum size=2.7mm,inner sep=0pt},
  movearrow/.style={-{Latex[length=2.4mm]},draw=CaltechOrange,very thick}
]
  \def\xstar{3}

  \fill[CaltechOrangeLight] (2.65,-0.35) rectangle (3.35,3.35);
  \node[stacklabel,above=4pt] at (\xstar,3.35) {slice $\{x\}\times Q_m$};

  \foreach \j/\lab in {0/z_1,1/z_2,2/\textcolor{blue}{y},3/z_3} {
    \draw[CaltechGray!45,thin] (-0.15,\j) -- (5.15,\j);
    \node[stacklabel,left=7pt] at (-0.15,\j) {$\lab$};
  }
  \foreach \i in {0,1,2,3,4,5} {
    \draw[CaltechGray!18,thin] (\i,-0.2) -- (\i,3.2);
  }
  \foreach \i in {0,1,2,3,4,5} {
    \foreach \j in {0,1,2,3} {
      \node[circle,draw=CaltechGray,fill=white,minimum size=2.4mm,inner sep=0pt] at (\i,\j) {};
    }
  }

  \node[stacklabel,below=7pt] at (2.5,-0.35) {$Q_a$ coordinate};
  \node[stacklabel,rotate=90] at (-1.25,1.5) {$Q_m$ coordinate};

  \node[pebble] at (0,0) {};
  \node[pebble] at (1,1) {};
  \node[pebble] at (5,3) {};

  \draw[movearrow] (0.25,0) -- (2.62,0);
  \draw[movearrow] (1.25,1) -- (2.62,1);
  \draw[movearrow] (4.75,3) -- (3.38,3);

  \node[pebble] at (\xstar,0) {};
  \node[pebble] at (\xstar,1) {};
  \node[pebble] at (\xstar,3) {};

  \node[circle,draw=blue,fill=blue!10,minimum size=5mm,inner sep=0pt,very thick] at (\xstar,2) {};

  \draw[-{Latex[length=2.4mm]},draw=blue,very thick] (\xstar,0.16) -- (\xstar,1.75);
  \draw[-{Latex[length=2.4mm]},draw=blue,very thick] (\xstar,1.18) -- (\xstar,1.75);
  \draw[-{Latex[length=2.4mm]},draw=blue,very thick] (\xstar,2.84) -- (\xstar,2.25);
\end{tikzpicture}
\caption{Schematic of the product step in $Q_n=Q_a\square Q_m$.  The horizontal
direction records the $Q_a$ coordinate and the vertical direction records the $Q_m$
coordinate.  The rows labeled $z_1,z_2,y,z_3$ represent fibers
$Q_a\times\{z\}$, while the highlighted column is the slice
$\{x\}\times Q_m$.  Black dots represent prepared piles.  Orange arrows indicate
moving piles within the fibers $Q_a\times\{z\}$ to first coordinate $x$; once
these piles are prepared in the highlighted slice, the blue arrows indicate the
pebbling sequence inside that copy of $Q_m$ toward the blue target $(x,y)$.}
\label{fig:product-step}
\end{figure}
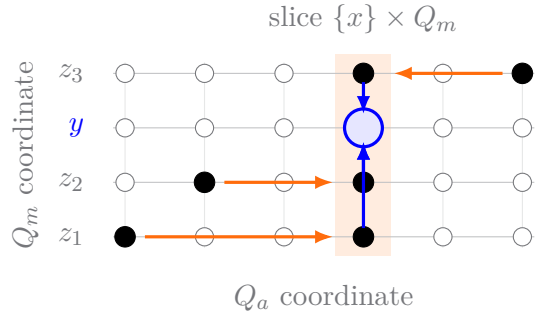

\begin{proof}[Proof of Lemma~\ref{lem:product-step}]
For every occupied $z$ put in the fiber $Q_a\times\{z\}$ a copy of an $s_z$-demand
distribution on $Q_a$.  To solve a target $(x,y)$, first use the copy in
$Q_a\times\{z\}$ to move $s_z$ pebbles to $(x,z)$ for every occupied $z$.  At that point,
inside the slice $\{x\}\times Q_m$, the numbers of pebbles at the vertices $(x,z)$ are
at least the pile sizes of $E$ at the corresponding vertices $z$. A pebbling sequence in
$Q_m$ that solves $y$ from $E$ can therefore be executed edge-for-edge in the slice
$\{x\}\times Q_m$, placing a pebble at $(x,y)$.
\end{proof}

We now combine the high-demand lemma and the composition step into a recursion that
proves the upper bound.

\begin{theorem}[Upper bound]\label{thm:upper}
$o(Q_n)\,=\,O\!\left(\left(\frac43\right)^n\right).$
\end{theorem}
\begin{proof}
Let $\Adem,C$ be the constants from Lemma~\ref{lem:high-demand}. In the recursion, the
smaller factor $Q_m$ supplies the demand values $s_z$, while the larger factor $Q_a$
supplies those demands inside the fibers.  The inductive pile-size lower bound will give
$s_z\ge 2^{\lfloor m/5\rfloor}$, so $m$ must be large enough for
$2^{\lfloor m/5\rfloor}$ to exceed the high-demand threshold in dimension $a$.
Choose $K>5\Adem\sqrt 2$.  Then $m/5$ dominates
$\Adem\sqrt{n\log(2n^2)}$ for large $n$; the floor only changes constants.
Choose $N_0$ so large
that, for every $n\ge N_0$, with
\[
m\,=\,\left\lceil K\sqrt{n\log n}\right\rceil,\qquad a\,=\,n-m,
\]
we have $m\le n/4$,
\[
2^{\lfloor m/5\rfloor}\,\ge\, C n^5 2^{\Adem\sqrt{n\log(2n^2)}},
\]
and $a\ge n_0$ and $a^{-3}\le n^{-2}\le 1/2$, so the high-demand lemma applies in
dimension $a$ with $\delta=n^{-2}$.
\label{def:parameters}

We define solvable distributions $D_n$ recursively.  For every $n<N_0$, choose any
solvable distribution $D_n$ on $Q_n$ whose occupied piles have size at least
$2^{\lfloor n/5\rfloor}$; for example, place $2^{\lfloor n/5\rfloor}$ pebbles on every
vertex.  For $n\ge N_0$, assume
that $D_m$ has already been constructed for
$m=\left\lceil K\sqrt{n\log n}\right\rceil$.
For each occupied vertex $z$ of $D_m$, let $s_z=D_m(z)$.  By induction every occupied
pile of $D_m$ satisfies $s_z\ge 2^{\lfloor m/5\rfloor}$, and by the choice of
$K,N_0$ this is at least $C\delta^{-2}a2^{\Adem\sqrt{a\log(2/\delta)}}$.
Thus Lemma~\ref{lem:high-demand} applies on the $Q_a$ factor for every required demand
$s_z$.  Define $D_n$ on $Q_n=Q_a\square Q_m$ by putting, in the fiber $Q_a\times\{z\}$, an
$s_z$-demand distribution on $Q_a$ of size at most $(1+n^{-2})s_z\left(\frac43\right)^a$.
By Lemma~\ref{lem:product-step} the distribution $D_n$ is solvable, and
\[
|D_n|\,\le\, (1+n^{-2})|D_m|\left(\frac43\right)^a.
\]
The pile-size invariant is preserved: the high-demand lemma uses piles of size at least
$2^{\lfloor a/3\rfloor}$, and since $a\ge 3n/4$ we have
$\lfloor a/3\rfloor\ge n/4-1\ge n/5\ge \lfloor n/5\rfloor$ whenever $n\ge 20$.
\label{lem:recurrence}

Put $R_n=|D_n|/(4/3)^n$.  The construction yields $R_n\le (1+n^{-2})R_m$, and for the
fully integral recursive bound we use the harmless weaker recurrence
$R_n\le (1+2n^{-2})R_m$, which absorbs the ceiling in the integer fiber multiplier and
changes only the final absolute constant.  Iterate this recurrence while $n_i\ge N_0$,
where $n_{i+1}=\left\lceil K\sqrt{n_i\log n_i}\right\rceil$.
By increasing $N_0$ we may assume $n_{i+1}\le n_i/4$ throughout the iteration, so that
$n_{i+1}^{-2}\ge 16\,n_i^{-2}$ and the sequence $(n_i^{-2})$ grows at least geometrically.
Summing this geometric series, $\sum_i n_i^{-2}\le n_{i_*}^{-2}\cdot\frac{16}{15}\le
\frac{16}{15N_0^2}$, where $n_{i_*}\ge N_0$ is the last active term in the iteration.
At the terminal dimension $b<N_0$, the base construction gives
$R_b\le 2^{\lfloor b/5\rfloor}\left(\frac32\right)^b$, so all terminal costs are bounded by
$R_{\mathrm{base}}=\sum_{j=0}^{N_0-1}2^{\lfloor j/5\rfloor}\left(\frac32\right)^j$.  Hence
\[
R_n\,\le\, R_{\mathrm{base}}\prod_i(1+2n_i^{-2})
\,\le\, R_{\mathrm{base}}\exp\!\left(2\sum_i n_i^{-2}\right)
\,\le\, R_{\mathrm{base}}\exp\!\left(\frac{4}{N_0^2}\right).
\]
\label{lem:loss}
This is an absolute constant, and hence $o(Q_n)\le |D_n|=O\!\left(\left(\frac43\right)^n\right)$.
\end{proof}

Theorem~\ref{thm:main} follows by combining Proposition~\ref{prop:lower} and
Theorem~\ref{thm:upper}.

\section*{Cooperation is necessary}

The upper bound is essentially cooperative: a matching upper bound cannot be obtained
from isolated stacks.

\begin{proposition}[Single-stack covering lower bound]\label{prop:covering}
If every target must be reached from a single stack of pebbles, then the stack centers
form a covering code and the total cost is at least
\[
\min_r 2^r\frac{2^n}{B(n,r)}
\,=\,\Omega\!\left(\sqrt n\left(\frac43\right)^n\right),
\qquad
B(n,r)\,=\,\sum_{i=0}^r\binom ni.
\]
\end{proposition}
\begin{proof}
If every target is reached from a single stack of $2^r$ pebbles, then the stack centers
form a radius-$r$ covering code, and the cost is at least $2^r 2^n/B(n,r)$.  Allowing
stacks of different sizes does not improve this, since a stack of radius $r$ covers at most
$B(n,r)$ targets at cost $2^r$; thus any single-stack covering strategy has cost at least
$\min_r 2^r 2^n/B(n,r)$.  The minimum occurs at $r\sim n/3$, where Stirling's formula gives
$\Omega\!\left(\sqrt n(4/3)^n\right)$.
\end{proof}

This covering-code formulation appears in Moews~\cite{Moews}; see~\cite{CoveringCodes}
for background on covering codes.  The bound exceeds Theorem~\ref{thm:main} by a factor of
order $\sqrt n$, so any upper bound matching it to within a constant factor must use
cooperation between stacks.

The construction avoids this loss by allowing different stacks to combine.  In the
product step, many fibers first prepare prescribed numbers of pebbles at the correct
$Q_a$-coordinate, and then those prepared piles are merged according to a pebbling
sequence in the smaller cube $Q_m$. Thus, a target is supplied not by a single stack, but
by a coordinated collection of stacks whose contributions coalesce.  In this precise
sense, cooperation is necessary to achieve the optimal asymptotic order.
This raises the question: for which graphs is cooperation necessary in this sense?

\section*{Formal verification}

The proof has been formalized in Lean 4, using Mathlib.  The formalization includes the
definitions of graph pebbling, pebbling moves, reachability, solvability, and the
relational predicate for an optimal pebbling number; the weight-function lower bound; the
probabilistic high-demand construction; the product-recursion step; and the final
asymptotic upper bound.

The Lean project consists of 20 Lean source files and 12,071 lines of Lean code.  The
project was built with Lean \texttt{v4.30.0-rc2} via Lake.  The project files contain no
\texttt{sorry}, \texttt{admit}, or custom \texttt{axiom} declarations.  At the
proof-producing level, the formalization proves the following certified version of the
main asymptotic result.

\begin{theorem}[Certified optimal pebbling bounds]\label{thm:certified}
Let $k$ be any integer satisfying the formal predicate that $k$ is the optimal
pebbling number of $Q_n$. Then
\[
\left(\frac43\right)^n \,\le\, k
\qquad\text{and}\qquad
k \,\le\, C_{\mathrm{Lean}}\left(\frac43\right)^n.
\]
A paper-facing wrapper also defines a noncomputable optimal pebbling number and derives
the same two-sided bounds directly for it.
\end{theorem}
The explicit certified constant in the formal upper bound is
\[
C_{\mathrm{Lean}}
\,=\,
\exp(4/N_0^2)
\sum_{j=0}^{N_0-1}2^{\lfloor j/5\rfloor}
\left(\frac32\right)^j,
\qquad
N_0\,=\,36\,329\,454\,321\,664\,=\,(128\cdot 217^2)^2.
\]
Numerically, $C_{\mathrm{Lean}}\approx 10^{8\,584\,550\,447\,692.749}$.
\label{const:clean}
This constant is not optimized; it is designed to be sufficiently large to provide a
concrete set of parameters for which formal estimates are verified.

The Lean development is a machine-readable version of the proof in this paper.
As argued by Booeshaghi, Luebbert, and Pachter~\cite{BLP}, machine-readable
scientific artifacts are a useful complement to human-readable exposition:
the former can be checked, searched, and composed by software, while the latter
can present the motivation and mathematical ideas in a form suited for people.
For a formal proof to serve this role, however, it is not enough for the Lean
project to build.  One must also know how the formal objects correspond to the
definitions, lemmas, propositions, and theorems in the paper.

To make this correspondence explicit, we developed \texttt{span}.  The tool reads
the LaTeX source of the paper, using ordinary \verb|\label| commands as the
paper-side object identifiers, and indexes the Lean source directly.  On the
paper side, theorem-like environments provide statements, titles, and kinds,
while inline labels record shorter mathematical objects that are naturally
presented in prose.  On the Lean side, declarations are indexed by fully
qualified names, and optional comments may record paper-facing hints, but the
alignment does not depend on a separate blueprint file.  A curated ledger is the
durable bridge between the two sources: it records which paper objects are
realized by which Lean declarations, and it allows many-to-one and one-to-many
correspondences when the human exposition and formal proof have different
granularities.

This ledger is deliberately auditable.  It stores the paper leg, the Lean leg,
the object kind, and the resulting status of each alignment span.  A coherent
span means that the named paper object and Lean declarations exist and have
compatible kinds; a paper half-span records a mathematical object present in the
paper but not formalized; a Lean half-span records formal material with no
paper-facing counterpart; and an incoherent span records a broken or incompatible
alignment.  Thus the ledger is not only a build artifact, but also a concise
accounting of the relation between the proof as read by a human and the proof as
checked by Lean.

Running \texttt{span} checks that every named paper object and every named Lean
declaration still exists, verifies kind compatibility, reports paper-only and
Lean-only objects, and compares the paper order with a dependency-respecting
order.  It can also emit JSON spines for the paper and Lean source, summarize
coverage statistics, compare a current ledger with a previous one, and draw
overlay graphs such as Figure~\ref{fig:span-dependencies}.  These checks are
designed to be deterministic: after an initial human or LLM-assisted alignment,
subsequent edits to either the paper or the Lean project can be checked without
re-solving the semantic matching problem.

Related tools address different parts of this workflow.  LeanBlueprint builds
formalization blueprints from annotated LaTeX sources~\cite{LeanBlueprint};
LeanArchitect extracts blueprint data from annotated Lean code and generates
synchronized LaTeX blueprint content~\cite{LeanArchitect}.  \texttt{span} is
complementary: it records and checks the alignment between an ordinary paper and
a Lean development, using labels and declaration names as stable endpoints.

The same mechanism applies beyond a single paper file and a single Lean file.
A \texttt{span} project may list several LaTeX roots and several Lean modules,
so sections, appendices, or related papers can share one ledger.  Each entry
names its paper source by label and its formal source by declaration name, so
moving a proof to another Lean module or reorganizing the LaTeX file changes the
location of the object but not its identity.  This makes the alignment robust
under ordinary editing, while still exposing real drift, such as a renamed
theorem, a deleted label, or a formal lemma whose role in the paper has changed.
In practice the workflow is small: \texttt{span build} produces the paper and
Lean spines, \texttt{span check} validates the ledger and reports drift, and
\texttt{span graph} renders the aligned dependency graph.  The generated files
are machine-readable; the ledger remains a reviewable source file edited
alongside the mathematical text and Lean code.  Here \texttt{span} finds 17
paper objects and an 18-entry ledger: 16 coherent alignment spans, one paper
half-span corresponding to Proposition~\ref{prop:covering}, one Lean half-span
for formal bookkeeping for the concrete upper bound, and no incoherent entries.

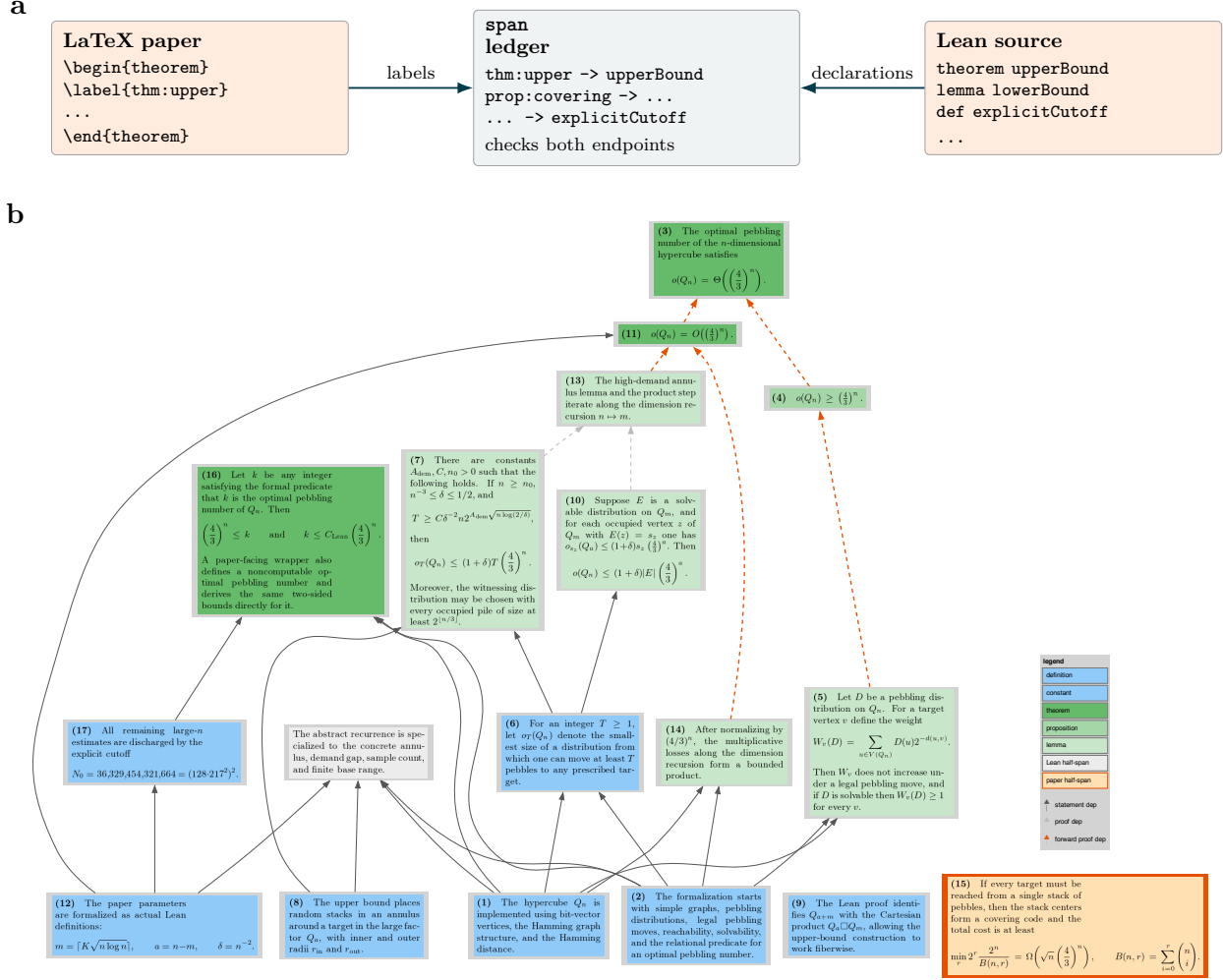
\begin{figure}[htbp]
\centering
\noindent\makebox[0pt][l]{\raisebox{2.05cm}{\textbf{a}}}\hspace{0.55cm}\resizebox{0.96\textwidth}{!}{%
\begin{tikzpicture}[
  font=\footnotesize,
  box/.style={
    draw=CaltechGray!70,
    rounded corners=2pt,
    fill=white,
    align=left,
    inner sep=5pt,
    minimum height=2.05cm
  },
  sourcebox/.style={box,fill=CaltechOrangeLight,text width=4.2cm},
  spanbox/.style={box,fill=CaltechDark!6,text width=4.65cm},
  arrow/.style={-{Latex[length=2.4mm]},draw=CaltechDark,thick}
]
  \node[sourcebox] (tex) at (0,0) {
    \textbf{LaTeX paper}\\[2pt]
    \ttfamily\scriptsize
    \textbackslash begin\{theorem\}\\
    \textbackslash label\{thm:upper\}\\
    ...\\
    \textbackslash end\{theorem\}
  };

  \node[spanbox] (span) at (6.7,0) {
    \textbf{\texttt{span}}\\[-1pt]
    \textbf{ledger}\\[2pt]
    \ttfamily\scriptsize
    thm:upper -> upperBound\\
    prop:covering -> ...\\
    ... -> explicitCutoff\\[1pt]
    \normalfont\scriptsize checks both endpoints
  };

  \node[sourcebox] (lean) at (13.4,0) {
    \textbf{Lean source}\\[2pt]
    \ttfamily\scriptsize
    theorem upperBound\\
    lemma lowerBound\\
    def explicitCutoff\\
    ...
  };

  \draw[arrow] (tex.east) -- node[above,font=\scriptsize] {labels} (span.west);
  \draw[arrow] (lean.west) -- node[above,font=\scriptsize] {declarations} (span.east);
\end{tikzpicture}
}

\vspace{0.75em}
\noindent\makebox[\textwidth][l]{\textbf{b}}\vspace{-1.1em}
\includegraphics[width=\textwidth,height=0.54\textheight,keepaspectratio]{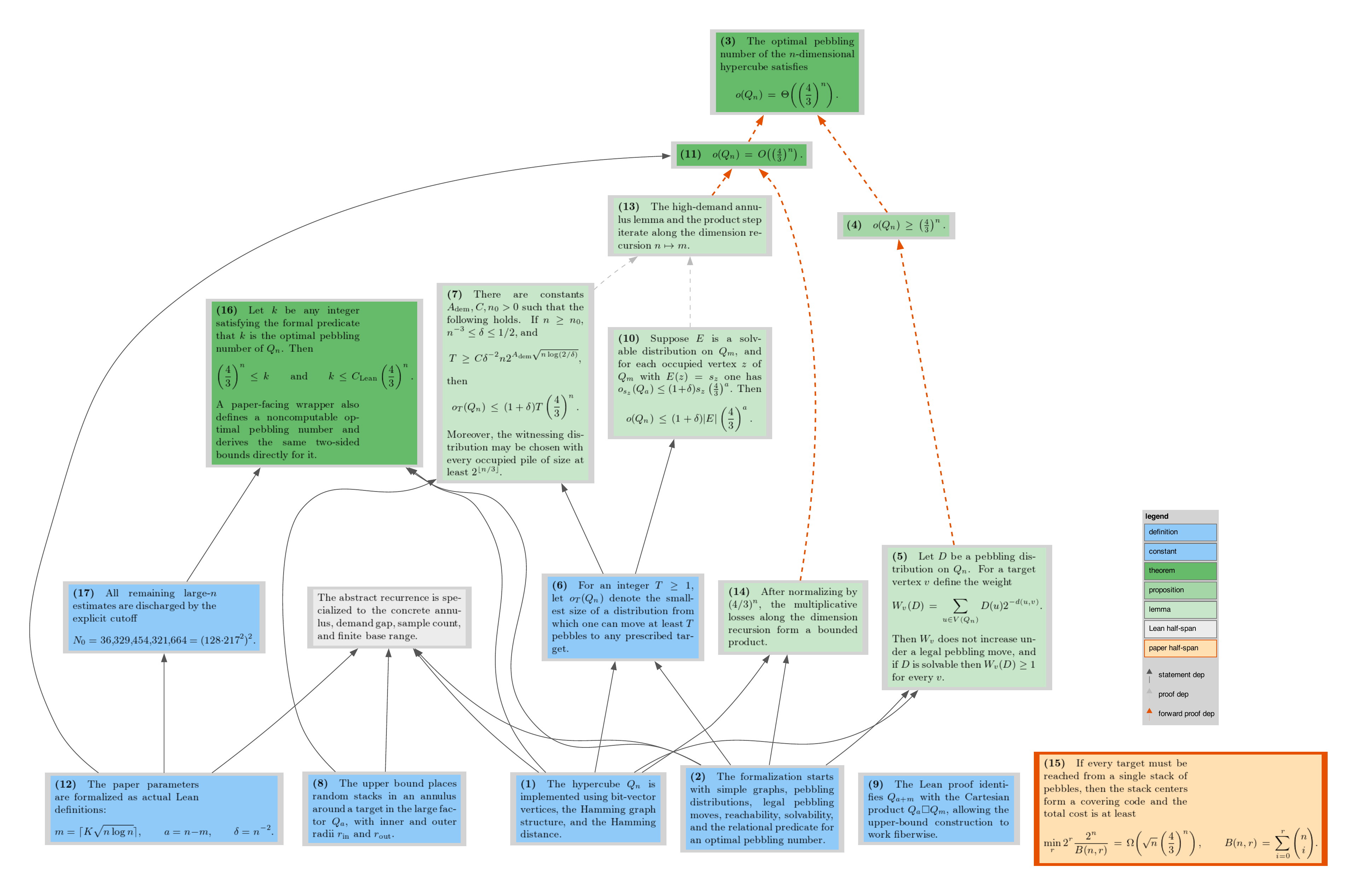}
\caption{The \texttt{span} alignment and proof graph.  (a) The ledger maps
LaTeX labels to Lean declarations; \texttt{span} checks that both sides remain
present and coherent.  (b) The paper-facing graph checked by \texttt{span}.
Nodes are paper objects ordered as a linear extension of the Lean dependency
graph, with LaTeX statements as labels.  A solid edge (\emph{statement dep})
\(Q\to P\) means that the Lean statement owned by \(P\) uses a declaration owned
by \(Q\), so \(P\) cannot be formulated without \(Q\).  A dashed gray edge
(\emph{proof dep}) records a proof reference: \(P\)'s proof cites \(Q\), but its
statement does not depend on \(Q\).  A dashed orange edge (\emph{forward proof
dep}) is such a proof reference whose prerequisite appears later in paper order,
hence a forward proof dependency rather than a statement-order violation.}
\label{fig:span-dependencies}
\end{figure}

\clearpage

\section*{Code and data availability}

The Lean formalization is available at
\url{https://github.com/pachterlab/P_2026_2}.  The \texttt{span} tool for
aligning LaTeX papers with Lean formalizations is available at
\url{https://github.com/pachterlab/span} and as the Python package
\url{https://pypi.org/project/span-lea/}.

\end{document}